\documentclass[12pt]{article}

\usepackage{amsmath}
\usepackage{amssymb}
\usepackage{float}
\usepackage{setspace}
\usepackage[english]{babel}
\usepackage[utf8]{inputenc}
\usepackage{amsthm}
\usepackage[makeroom]{cancel}
\usepackage{wasysym}
\usepackage{verbatim}
\usepackage{graphicx}
\usepackage[margin = 1in]{geometry}
\usepackage[colorinlistoftodos]{todonotes}

\usepackage{diagbox}

\theoremstyle{definition}
\newtheorem{definition}{Definition}
\theoremstyle{plain}
\newtheorem{theorem}{Theorem}
\newtheorem{lemma}[theorem]{Lemma}
\newtheorem{corollary}[theorem]{Corollary}
\newtheorem{proposition}[theorem]{Proposition}
\def\Z{\mathbb{Z}}

\def\eps{\varepsilon}

\begin{document}
%\begin{comment}
\author{Shilin Ma\\ Kevin McGown \\ Devon Rhodes \\ Mathias Wanner}
\title{Explicit bounds for small prime nonresidues}
\maketitle
%\end{comment}
\begin{center}
\end{center}
\vspace{2ex}

\begin{abstract}
Let $\chi$ be a Dirichlet character modulo a prime~$p$.
We give explicit upper bounds on $q_1<q_2<\dots<q_n$, the $n$ smallest prime nonresidues of $\chi$.
More precisely, given $n_0$ and $p_0$ there exists an absolute constant $C=C(n_0,p_0)>0$
such that $q_n\leq Cp^{\frac{1}{4}}(\log p)^{\frac{n+1}{2}}$
whenever $n\leq n_0$ and $p\geq p_0$.
\end{abstract}

\section{Introduction}
Let $\chi$ be a nonprincipal Dirichlet character modulo a prime~$p$.
If $\chi(n)\notin\{0,1\}$, then we refer to $n$ as a nonresidue of $\chi$.
Let $q_1<q_2<\dots<q_\ell$ denote the $\ell$ smallest prime nonresidues of~$\chi$.
Giving an upper bound on $q_1$ is an important classical problem that has received much attention.
Indeed, in the case of the Legendre symbol,
$q_1$ is the least quadratic nonresidue mod $p$.
In~1963, Burgess showed that for each $\eps>0$, one has $q_1\ll p^{\frac{1}{4\sqrt{e}}+\eps}$ (see~\cite{bib:Burgess1,bib:Burgess2}),
and this result has stood as the state of the art since this time, save a recent improvement to the ``$\eps$'' in the quadratic case (see~\cite{bib:Banks.Guo}).
In~2015,
Pollack proved the following result (see~\cite{bib:Pollack}):  For each $\eps>0$, there are numbers
$m_0=m_0(\eps)$ and $\kappa=\kappa(\eps)>0$ such that
for all $m>m_0$ and each nonprincipal character $\chi$ modulo $m$, there are more than
$m^\kappa$ prime nonresidues of $\chi$ not exceeding $m^{\frac{1}{4\sqrt{e}}+\eps}$.
In particular, for all $\eps>0$ and all $k$, one has $q_k\ll p^{\frac{1}{4\sqrt{e}}+\eps}$,
although this hides the dependence on $k$ and $\eps$.
As we alluded to a moment ago, Banks and Guo have recently shown that
$q_k \ll p^{\frac{1}{4\sqrt{e}}}\exp(\sqrt{e^{-1}\log p\log\log p})$
in the case where $\chi$ is the Legendre symbol (see~\cite{bib:Banks.Guo}),
provided $k\leq p^{\frac{1}{8\sqrt{e}}}\exp(\frac{1}{2}\sqrt{e^{-1}\log p\log\log p}-\frac{1}{2}\log\log p)$.

Often in applications (see, for example,~\cite{bib:McGown2,bib:McGown.Lezowski,bib:Pollack.Trevino,bib:Booker})
one requires estimates that are completely explicit,
and one is willing to accept a weaker asymptotic in order to obtain constants of a reasonable magnitude.
Our goal here is to give an explicit upper bound on $q_k$, the $k$th smallest prime nonresidue.
Naturally, our upper bounds are
asymptotically weaker than those given in~\cite{bib:Pollack} and~\cite{bib:Banks.Guo}.
The following is our main result from which one can easily derive bounds of the desired form.
\begin{theorem}\label{T:1}
Let $\chi$ be a nonprincipal Dirichlet character modulo~$p$.
Let $u\in\Z^+$ be squarefree and assume all its prime factors are less than $p$.
Set $n=\omega(u)+1$ where $\omega(u)$ is the number of distinct prime factors of $u$.
Suppose that $\chi(m)=1$
whenever $0<m\leq H<p$ and $(m,u)=1$. 
%Let $n-1$ be the number of distinct prime factors of $u$.
Then
\[
  H\leq p^{\frac{1}{4}}(\log p)^{\frac{n+1}{2}}g(n,p)
\]
where
\[
g(n,p) =\frac{\pi}{3}\sqrt{2e}\left(\frac{n}{n+1}\right)\left(\frac{1+\frac{\sqrt{2}}{2B\log{p}-3}}{1-\frac{\pi^2}{9}\frac{\log{X^*}+9}{3X^*}}\right)^{\frac{1}{2}}
\]
using
\[
X^*=\frac{\pi}{3}(2e)^{-\frac{1}{2}}\left(\frac{n+1}{n}\right)^{n-1}\frac{p^\frac{1}{4}}{(\log p)^{\frac{n-1}{2}}}\cdot
%\frac{1}{1+\frac{n+1}{n}e^{\frac{-1}{n}}(\log p)^{-1}}
\left(1-\frac{n+1}{n}e^{-1/n}(\log p)^{-1}\right)
\,,\qquad
B=\frac{n}{2(n+1)}
\,,
\]
provided $X^*>3.8$ and
%$p>\max\{2\cdot 10^6,\,\exp(8(n-1))\}$.
$\log p>\max\{\exp(8/3),\,8(n-1)\}$.
\end{theorem}
%The following result which is more easily applied and follows immediately from the theorem and a small amount of calculation.
\begin{corollary}\label{C:1}
Fix two integer constants $p_0$ and $n_0$ such that $X^*(p_0,n_0)>3.8$ and $p_0>\max\{2\cdot 10^6,\exp(8(n-1))\}$.
Then there exists an explicit constant $C=g(n_0,p_0)$ such that
\[
  q_n\leq Cp^{\frac{1}{4}}(\log p)^{\frac{n+1}{2}}
\]
for all $p\geq p_0$ and $n\leq n_0$.
\end{corollary}

%\begin{table}[H]
%\centering
%\caption{Constants for various $p_0$ and $n_0$}
%\label{my-label}
%\begin{tabular}{|c||c|c|c|c|c|c|c|c|c|}
%\hline
%\diagbox{$n_0$}{$p_0$} & $10^7$ & $10^8$ & $10^9$ & $10^{10}$ & $10^{15}$ & $10^{20}$ & $10^{25}$ & $10^{30}$ & $10^{35}$ \\ \hline\hline
%$1$ & $1.604$ & $1.462$ & $1.391$ & $1.350$ & $1.282$ & $1.264$ & $1.254$ & $1.248$ & $1.244$\\ \hline
%$2$ & $2.521$ & $2.108$ & $1.926$ & $1.829$ & $1.693$ & $1.670$ & $1.661$ & $1.655$ & $1.651$\\ \hline
%$3$ & $-$ & $11.61$ & $3.197$ & $2.500$ & $1.927$ & $1.876$ & $1.864$ & $1.858$ & $1.854$\\ \hline
%$4$ & $-$ & $-$ & $-$ & $-$ & $2.234$ & $2.014$ & $1.987$ & $1.980$ & $1.976$\\ \hline
%$5$ & $-$ & $-$ & $-$ & $-$ & $6.870$ & $2.200$ & $2.079$ & $2.063$ & $2.058$\\ \hline
%$6$ & $-$ & $-$ & $-$ & $-$ & $-$ & $3.406$ & $2.206$ & $2.128$ & $2.116$\\ \hline
%$7$ & $-$ & $-$ & $-$ & $-$ & $-$ & $-$ & $2.923$ & $2.222$ & $2.165$\\ \hline
%$8$ & $-$ & $-$ & $-$ & $-$ & $-$ & $-$ & $-$ & $2.748$ & $2.240$\\ \hline
%\end{tabular}
%\end{table}

\begin{table}[H]
\centering
\caption{Constants for various $p_0$ and $n_0$}
\label{my-label}
\begin{tabular}{|c||c|c|c|c|c|c|c|c|c|}
\hline
\diagbox{$n_0$}{$p_0$} & $10^7$ & $10^8$ & $10^9$ & $10^{10}$ & $10^{15}$ & $10^{20}$ & $10^{25}$ & $10^{30}$ & $10^{35}$ \\ \hline\hline
$1$ & $1.530$ & $1.433$ & $1.378$ & $1.344$ & $1.282$ & $1.264$ & $1.254$ & $1.248$ & $1.244$\\ \hline
$2$ & $2.408$ & $2.070$ & $1.909$ & $1.821$ & $1.692$ & $1.670$ & $1.661$ & $1.655$ & $1.651$\\ \hline
$3$ & $-$ & $7.170$ & $3.087$ & $2.468$ & $1.926$ & $1.876$ & $1.864$ & $1.858$ & $1.854$\\ \hline
$4$ & $-$ & $-$ & $-$ & $-$ & $2.230$ & $2.014$ & $1.987$ & $1.980$ & $1.976$\\ \hline
$5$ & $-$ & $-$ & $-$ & $-$ & $6.469$ & $2.198$ & $2.079$ & $2.063$ & $2.058$\\ \hline
$6$ & $-$ & $-$ & $-$ & $-$ & $-$ & $3.386$ & $2.205$ & $2.128$ & $2.116$\\ \hline
$7$ & $-$ & $-$ & $-$ & $-$ & $-$ & $-$ & $2.916$ & $2.222$ & $2.165$\\ \hline
$8$ & $-$ & $-$ & $-$ & $-$ & $-$ & $-$ & $-$ & $2.745$ & $2.240$\\ \hline
\end{tabular}
\end{table}

To our knowledge, the previous corollary constitutes the first explicit upper bound on $q_k$ when $k\geq 3$.
When $k=1$, there is the work of Norton (see~\cite{bib:Norton}) that was later superceded by Trevi\~no (see~\cite{bib:Trevino1})
and when $k=2$ there is a paper by McGown (see~\cite{bib:McGown1}).
The proof of our result involves a modification of McGown's work (see~\cite{bib:McGown1}), which is based on the method of Burgess (see~\cite{bib:Burgess1, bib:Burgess2}), and the adoption of Trevi\~no's results (see~\cite{bib:Trevino2}).

%\begin{table}
%???
%\end{table}

%Lemma \ref{upperbound} gives a upperbound of the sum $S(\chi,h,r)= \sum_{x=0}^{p-1} |\sum_{m=0}^{h-1} \chi(x+m) | ^{2r}$, whereas Proposition \ref{Prop:S_lower} establishes a lowerbound for that same sum in terms of $p,r,h$ and $H$, the length of an interval where every number is a residue if it is coprime to a certain product $u$. Combining the two, we obtain the desired upper bound on $H$. 

\section{Preparations}

\begin{lemma}\label{upperbound}
Let $\chi$ be a nonprincipal Dirichlet character to a prime modulus $p$. Let $h,r$ be positive integers
satisfying $h<p$ and $r\leq 9h$. Then
\[S(\chi,h,r)= \sum_{x=0}^{p-1} \left|\sum_{m=0}^{h-1} \chi(x+m) \right| ^{2r}\leq \sqrt{2}\left(\frac{2r}{e}\right)^rph^r+(2r-1)p^{\frac{1}{2}}h^{2r}\,.\]
\end{lemma}

\begin{proof}
From Theorem 1.1 of \cite{bib:Trevino1} we have
\[
S(\chi,h,r) \leq \frac{(2r)!}{2^rr!}ph^r+(2r-1)p^{\frac{1}{2}}h^{2r}
\,.
\]
Using the explicit version of Stirling's formula given in \cite{bib:Robbins}, we have
\begin{align*}
\frac{(2r)!}{2^rr!}\leq \frac{\sqrt{2\pi}(2r)^{2r+\frac{1}{2}}e^{-2r}e^{\frac{1}{24r}}}{2^r\sqrt{2\pi}r^{r+\frac{1}{2}}e^{-r}e^\frac{1}{12r+1}}\leq \sqrt{2}\left(\frac{2r}{e}\right)^r
\,.
\end{align*}
\end{proof}

\begin{definition}
Let $0 \leq b < a \leq X$ with $(a,b)=1$.
For constants $H$, $h$, define the following intervals:
\begin{align*}
&\mathcal{I}(a,b):=\left(\frac{bp}{a},\frac{bp+H}{a}\right],   
&\mathcal{I}(a,b)^*:=\left(\frac{bp}{a},\frac{bp+H}{a}-h+1\right]\,,\\
&\mathcal{J}(a,b):=\left[\frac{bp-H}{a},\frac{bp}{a}\right),   
&\mathcal{J}(a,b)^*:=\left[\frac{bp-H}{a},\frac{bp}{a}-h+1\right)
\,.
\end{align*}

\end{definition}

\begin{lemma}\label{Disjoint}
Let $X$ be a real number such that $2XH<p$. Then the intervals
$\mathcal{I}(a,b)$, $\mathcal{J}(a,b)$ where $0 \leq b < a \leq X$ and $(a,b)=1$ are disjoint subintervals of $(0,p-H)$, except for $\mathcal{J}(1,0)=[-H,0)$.
\end{lemma}

\begin{proof}
This is Lemma 2 of \cite{bib:McGown1}.
\end{proof}

\begin{lemma}\label{Sum_chi}
Let $h,u_1,u_2\in\Z^+$ where $u_1=q_{1,1}q_{1,2}\ldots q_{1,k}$ with each prime $q_{1,i} < h$ and $u_2=q_{2,1}q_{2,2}\ldots q_{2,j}$ with each prime $h \leq q_{2,i}<p$. Suppose $\chi$ is Dirichlet character modulo~$p$ such that
$\chi (n) = 1$ for all $n\in (0,H]$ with $(n,u_1u_2) = 1$. If $z\in \mathcal{I}^*(a,b)\cup\mathcal{J}^*(a,b)$ with $(a,b) = 1$ and $u_1|a$, then 
\[
\left|\sum_{m=0}^{h-1} \chi(z+m)\right|\geq h- 2j 
\,.
\]
\end{lemma}
\begin{proof}
Suppose $z\in\mathcal{I}^*(a,b)$ with $(a,b)=1$ and $u_1|a$.
Then for all $0\leq n\leq h-1$, $z+n\in \mathcal{I}(a,b)$. Therefore $a(z+n)-bp \in (0,H]$. Since $\chi (z+n) = \bar{\chi} (a)\chi(a(z+n)-bp)$ and $a(z+n)-bp \in (0,H]$ and $(u_1,a(z+n)-bp)=1$, this will equal
%$\bar{\chi}(q)$
$\chi(z+n)=\bar{\chi}(a)$
if $(u_2,a(z+n)-bp) = 1$.
This can only fail once for each divisor $q_{2,i}$ of $u_2$.\\ 
\\To see this, suppose $q_{2,i}\mid a(z+n_1)-bp$ and $q_{2,i}\mid a(z+n_2)-bp$ for two different values $n_1,n_2 \in [0,h-1]$. Since $(a,b)=1$ we also know $q_{2,i}\nmid a$. We have $q_{2,i}\mid a(z+n_1)-bp-[a(z+n_2)-bp]$ and thus $q_{2,i}\mid(n_1-n_2)$. Now we have $h \leq q_{2,i} \leq n_1-n_2 \leq h-1$, which is a contradiction.

Application of the triangle inequality now gives the result; indeed, we have
\[
  \left|\sum_{m=0}^{h-1}\chi(z+m)\right|\geq\left|\sum_{m=0}^{h-1}\bar{\chi}(a)\right|
  -
  \sum_{m=0}^{h-1}\left|\chi(z+m)-\bar{\chi}(a)\right|
  \geq
  h-
  2j
  \,.
\]
The proof when $z\in\mathcal{J}^*(a,b)$ is similar.  In this case we have
$-a(z+n)+bp\in(0,H]$ which implies $\chi(z+n)=\bar{\chi}(-a)$
if $(u_2, a(z+n)-bp)=1$.  The result follows as before.
\end{proof}

\begin{lemma}\label{Trev}
Let $x>1$ be a real number. Then
\[2x\sum_{a \leq x}\frac{\phi(a)}{a}-\sum_{a \leq x}\phi(a) \geq \frac{9}{\pi^2}x^2f(x)\]
where \[f(x)=1-\frac{\pi^2}{9}\frac{\log x+9}{3x}
\,.
\]
\end{lemma}

\begin{proof}
This is Lemma 3.2 of \cite{bib:Trevino1}.
\end{proof}

\begin{proposition}\label{Prop:S_lower}
Let $h,r,u_1,u_2\in\Z^+$ with $u_1=q_{1,1}q_{1,2}\ldots q_{1,k}$ with each prime $q_{1,i} < h$ and $u_2 = q_{2,1}q_{2,2} \ldots q_{2,j}$ with each $h\leq q_{2,i} <p$. Suppose $\chi$ is a Dirichlet character modulo a prime~$p$ such that  $\chi (n) = 1$ for all $n\in[1,H]$ satisfying $(n,u_1u_2) = 1$. Assume $2h<H<(hp)^{\frac{1}{2}}$ and set $X:=\frac{H}{2h}>1$. Then 
\begin{align*}
S(\chi,h,r)\geq \frac{18}{\pi^2}h(h-2j)^{2r}\frac{\phi(u_1)}{u_1^2}X^2f\left(\frac{X}{u_1}\right)
\,,
\end{align*}
provided $X/u_1>1$.
\end{proposition}

\begin{proof}
First, observe that by Lemma~\ref{Disjoint} we have
\begin{align*}
S(\chi,h,r) &= \sum_{x=0}^{p-1} \left|\sum_{m=0}^{h-1} \chi(x+m) \right|^{2r} \\
&\geq \sum_{\substack{0\leq b< a\leq X \\ (a,b) = 1\\u_1|a}}
\sum_{z\in\mathcal{I}^*(a,b)\cup\mathcal{J}^*(a,b)} \left|\sum_{m=0}^{h-1}\chi(z+m) \right|^{2r} 
\,.
\end{align*}\\
\textrm{Applying Lemma~\ref{Sum_chi}, and noting that $\mathcal{I}^*(a,b)\cup\mathcal{J}^*(a,b)$ has at least $2(\frac{H}{a} - h)$ elements, we obtain}
\begin{align*}
S(\chi,h,r)&\geq\sum_{\substack{0\leq b <a\leq X \\ (a,b) = 1\\u_1|a}} 2\left(\frac{H}{a} - h\right) (h-2j)^{2r} 
\\
%\end{align*}
%\begin{align*}
&=2(h-2j)^{2r} \sum_{\substack{1\leq a\leq X\\u_1|a}} \left(\frac{H}{a} - h\right)\phi (a)\\
&=2(h-2j)^{2r} \sum_{t\leq \frac{X}{u_1}} \left( \frac{H}{tu_1} - h\right) \phi(tu_1)\\
&\geq 2(h-2j)^{2r} \sum_{t\leq \frac{X}{u_1}}\left( \frac{H}{tu_1} - h\right) \phi(t)\phi(u_1)\\
&=2(h-2j)^{2r}\phi(u_1)\left(\frac{H}{u_1}\sum_{t\leq \frac{X}{u_1}}\frac{\phi(t)}{t} - h\sum_{t\leq \frac{X}{u_1}}\phi(t)\right)
\,.
\end{align*}
Replacing $H$ with $2Xh$ the above is equal to
\[
2(h-2j)^{2r}h\phi(u_1)\left(\frac{2X}{u_1}\sum_{0\leq t\leq\frac{X}{u_1}}\frac{\phi(t)}{t} - \sum_{0\leq t\leq\frac{X}{u_1}}\phi(t)\right)
\,.
\]
Now we may apply Lemma \ref{Trev} to conclude
\begin{align*}
S(\chi,h,r) \geq \frac{18}{\pi^2}h(h-2j)^{2r}\frac{\phi(u_1)}{u_1^2}X^2f\left(\frac{X}{u_1}\right)
\,.
\end{align*}
\end{proof}
\begin{lemma}\label{convexity}
Let $h$ and $r$ be positive integers with $j \leq h/8$. Then
\[
\left(\frac{h}{h-2j}\right)^{2r} \leq \exp\left(\frac{16rj}{3h}\right)
\,.
\]
\end{lemma}

\begin{proof}
By the convexity of the logarithm, we know $\log t \leq t-1$ for all $t$.
It follows that, for $j\leq h/8$,
\[\log\left(\frac{h}{h-2j}\right) \leq \frac{h}{h-2j}-1 = \frac{2j}{h-2j}\leq\frac{8j}{3h}\]
which implies
\[2r\log\left(\frac{h}{h-2j}\right) \leq \frac{16rj}{3h}\,,\]
and therefore
\[\left(\frac{h}{h-2j}\right)^{2r} \leq \exp\left(\frac{16rj}{3h}\right)\,.\]
\end{proof}

%\begin{lemma}\label{B/A}
%If $A = \frac{2B}{e}\exp\left(\frac{1}{2B}\right)$,
%then $\left(\frac{2B}{Ae}\right)^{B\log p}= p^{-\frac{1}{2}}\,.$
%\end{lemma}
%
%\begin{proof}
%Trivial.
%\end{proof}

%\begin{theorem}\label{theorem}
%Let $\chi$ be a non principle Dirichlet character modulo $p$ and suppose $\chi(x)=1$ whenever $0<x\leq H$ and $(x,u)=1$. Let $n-1$ be the number of distinct prime factors of $u$. Then
%\[H \leq p^{\frac{1}{4}}(\log{p})^{\frac{n+1}{2}}g(n,p)\]
%where \[g(n,p) =\frac{\pi}{3}\sqrt{2e}\left(\frac{n}{n+1}\right)\left(\frac{1+\frac{\sqrt{2}e^{\frac{1}{2B}}}{2B\log{p}-3}}{1-\frac{\pi^2}{9}\frac{\log{X^*}+9}{3X^*}}\right)^{\frac{1}{2}}\]
%using \[X^*=\frac{\pi}{3}(2e)^{-\frac{1}{2}}\left(\frac{n+1}{n}\right)^{n-1}\frac{p^\frac{1}{4}}{(\log p)^{\frac{n-1}{2}}}\cdot\frac{1}{1+\frac{n+1}{n}e^{\frac{-1}{n}}(\log p)^{-1}}\]
%and \[B=\frac{n}{2(n+1)}\] provided $X^*>3.8$.
%\end{theorem}

\section{Proof of the main result}

\begin{proof}[Proof of Theorem~\ref{T:1}]
The conditions $X^*>3.8$ and $p>403$ guarantee that all the denominators in the expression for $g(n,p)$ are positive.
Notice also that the condition $\log p>8(n-1)$ implies that, in particular, $n\leq \frac{1}{4}\log p$.

Let $h = \lceil A \log p \rceil$ and $r = \lfloor B \log p \rfloor$ with $A=\frac{n}{n+1}\exp\left(\frac{1}{n}\right)$ and $B=\frac{n}{2(n+1)}$.
One verifies that $\log{p}>e^\frac{16B}{3A}$ is also satisfied.
Indeed, $\frac{2B}{A}=e^{-1/n}\leq 1$ and therefore the condition $\log p>\exp(8/3)$ suffices.
%this holds when $p>2\cdot 10^6$.

Write $X=H/(2h)$. Let $u_1$ be composed of the $k$ distinct prime factors of $u$ less than $h$ and let $u_2$ be composed of the $j$ prime factors greater than or equal to $h$ as in Proposition~\ref{Prop:S_lower}, so $j+k=n-1$. Note that our choices of $r$, $h$ will allow us to apply Lemma~\ref{convexity},
and moreover, we have $A = \frac{2B}{e}\exp\left(\frac{1}{2B}\right)$
from which it follows that
\begin{equation}\label{E:old.lemma}
\left(\frac{2B}{Ae}\right)^{B\log p}= p^{-\frac{1}{2}}
\,;
\end{equation}
both of these facts will be employed forthwith.

%Indeed, for $n \geq 8$ the condition on $X^*$ guarantees $\log{p}>57$, which is greater than the maximum of $e^\frac{B}{A}$, and the cases from 1 to 7 can be checked individually. Finally, if $n>\frac{1}{4}\log{p}$, then $X^*$ is less than 3.7, justifying our assumption that $n\leq \frac{1}{4}\log{p}$.

We may assume that $H>\frac{\pi}{3}\sqrt{2e}\left(\frac{n}{n+1}\right)p^{\frac{1}{4}}(\log{p})^\frac{n+1}{2}$ or there would be nothing to prove. Using this and $u_1\leq(h-1)^k$, we get $X^*$ as a lower bound for $\frac{X}{u_1}$;
indeed,
\begin{align*}
\frac{X}{u_1}
%&=
%\frac{H}{2h u_1}\\
&\geq
\frac{\frac{\pi}{3}\sqrt{2e}\left(\frac{n}{n+1}\right)p^{\frac{1}{4}}(\log p)^{\frac{n+1}{2}}}{2h(h-1)^k}
\\
&\geq
\frac{\frac{\pi}{3}\sqrt{2e}\left(\frac{n}{n+1}\right)p^{\frac{1}{4}}(\log p)^{\frac{n+1}{2}}(1-h^{-1})}{2(h-1)^{k+1}}
\\
&\geq
\frac{\frac{\pi}{3}\sqrt{2e}\left(\frac{n}{n+1}\right)p^{\frac{1}{4}}(\log p)^{\frac{n+1}{2}}
\left(1-\frac{n+1}{n}e^{-1/n}(\log p)^{-1}\right)}
{2\left(\frac{n}{n+1}e^{1/n}\log p\right)^n}
\\
&=
%\frac{A}{(\log p)^{\frac{n+1}{2}}}
%\frac{A\left(\frac{n+1}{n}\right)^{n-1}p^{\frac{1}{4}}}{(\log p)^{\frac{n+1}{2}}}
\frac{\pi}{3}\frac{1}{\sqrt{2e}}
\frac{\left(\frac{n+1}{n}\right)^{n-1}p^{\frac{1}{4}}}{(\log p)^{\frac{n-1}{2}}}
\left(1-\frac{n+1}{n}e^{-1/n}(\log p)^{-1}\right)
\,.
\end{align*}

Using the upper and lower bounds for $S(\chi,h,r)$
given in Lemma~\ref{upperbound} and Proposition~\ref{Prop:S_lower} respectively,
we find
\begin{align*}
\frac{18}{\pi^2}h(h-2j)^{2r}\frac{\phi(u_1)}{u_1^2}\left(\frac{H}{2h}\right)^2f\left(\frac{X}{u_1}\right) &\leq (2r-1)p^{\frac{1}{2}}h^{2r}\left(1+\frac{\sqrt{2}}{2r-1}\left(\frac{2r}{he}\right)^rp^{\frac{1}{2}}\right)
\end{align*}
which implies
\begin{align*}
\frac{18}{\pi^2}\frac{\phi(u_1)}{u_1^2}H^2f\left(\frac{X}{u_1}\right)&\leq 4h(2r-1)\left(\frac{h}{h-2j}\right)^{2r}p^{\frac{1}{2}}\left(1+\frac{\sqrt{2}}{2r-1}\left(\frac{2r}{he}\right)^rp^{\frac{1}{2}}\right) \\
&\leq 4(h-1)(2r)e^{\frac{16rj}{3h}}p^{\frac{1}{2}}\left(1+\frac{\sqrt{2}}{2r-1}\left(\frac{2r}{he}\right)^rp^{\frac{1}{2}}\right)
\,.
\end{align*}
Substituting our values for $h$ and $r$ and using the fact $\log{p}\geq e^{\frac{16B}{3A}}$,
together with~(\ref{E:old.lemma}), gives us
\begin{align*}
\frac{18}{\pi^2}\frac{\phi(u_1)}{u_1^2}H^2f\left(\frac{X}{u_1}\right)&\leq 4(A \log p)(2B\log p)e^{\frac{16Bj}{3A}}p^{\frac{1}{2}}\left(1+\frac{\sqrt{2}}{2B\log p-3}\left(\frac{2B}{Ae}\right)^{B\log p}p^\frac{1}{2}\right) \\
&= 4(A\log{p})(2B\log{p})e^{\frac{16Bj}{3A}}p^{\frac{1}{2}}\left(1+\frac{\sqrt{2}}{2B\log{p}-3}\right) \\
&\leq 8AB(\log{p})^2p^{\frac{1}{2}}(\log{p})^j\left(1+\frac{\sqrt{2}}{2B\log{p}-3}\right)
\,.
\end{align*}

Since $u_1$ is composed of $k$ prime factors strictly less than $h$, we have
\[
\frac{\phi(u_1)}{u_1^2}\geq \left(\frac{h-2}{(h-1)^2}\right)^k \geq \left(\frac{A\log{p}-2}{(A\log{p})^2}\right)^k
\,.
\]
Using this, and substituting our values of $A$ and $B$, we find
\begin{align*}
\frac{18}{\pi^2}H^2f\left(\frac{X}{u_1}\right) &\leq 8AB(\log{p})^{2}(\log{p})^jp^{\frac{1}{2}}\left(\frac{(A\log{p})^2}{A\log{p}-2}\right)^k\left(1+\frac{\sqrt{2}}{2B\log{p}-3}\right) \\
&\leq 8A^{k+1}B(\log{p})^{k+2}(\log p)^jp^{\frac{1}{2}}\left(\frac{A\log{p}}{A\log{p}-2}\right)^k\left(1+\frac{\sqrt{2}}{2B\log{p}-3}\right)\\
% &\leq 8\left(\frac{n}{e(n+1)}e^{\frac{n+1}{n}}\right)^n\frac{n}{2(n+1)}\left(\log{p}\right)^{k+j+2}p^{\frac{1}{2}}\left(\frac{A\log{p}}{A\log{p}-2}\right)^k\left(1+\frac{\sqrt{2}e^{\frac{1}{2B}}}{2B\log{p}-3}\right) \\
&\leq 8\left(\frac{n}{n+1}e^\frac{1}{n}\right)^{n}\frac{n}{2(n+1)}\left(\log{p}\right)^{n+1}p^{\frac{1}{2}}\left(\frac{A\log{p}}{A\log{p}-2}\right)^k\left(1+\frac{\sqrt{2}}{2B\log{p}-3}\right) \\
&\leq 4e\left(\frac{n}{n+1}\right)^{n+1}\left(\log{p}\right)^{n+1}p^{\frac{1}{2}}\left(\frac{A\log{p}}{A\log{p}-2}\right)^k\left(1+\frac{\sqrt{2}}{2B\log{p}-3}\right)
\,.
\end{align*}
Since $n\leq \frac{1}{4}\log{p}$, we can show \[\frac{n}{n+1}\cdot\frac{A\log{p}}{A\log{p}-2}\leq \frac{\log{p}}{\log{p}+4}\cdot\frac{A\log{p}}{A\log{p}-2}=\frac{A(\log{p})^2}{A(\log{p})^2+(4A-2)\log{p}-8}\,.\]
This is less than 1 whenever $\log{p}>4$.
Since $k \leq n-1$ we can use this to drop some terms from the product, which yields
\[\frac{18}{\pi^2}H^2f\left(\frac{X}{u_1}\right)\leq 4e\left(\frac{n}{n+1}\right)^2(\log{p})^{n+1}p^{\frac{1}{2}}\left(1+\frac{\sqrt{2}}{2B\log{p}-3}\right)\,.\]

Isolating $H^2$, and noting that $f(\frac{X}{u_1}) \geq f(X^*)$, gives
\begin{align*}
H^2 &\leq\frac{2e\pi^2}{9}\left(\frac{n}{n+1}\right)^2(\log{p})^{n+1}p^{\frac{1}{2}}\left(1+\frac{\sqrt{2}}{2B\log{p}-3}\right) \frac{1}{f\left(X^*\right)}
\,.
\end{align*}
Taking the square root of both sides and rearranging gives the desired result.
%\[H \leq p^{\frac{1}{4}}(\log{p})^{\frac{n+1}{2}}g(n,p).\]
\end{proof}

%\textbf{Corollary 1}\\
%Fix two integer constants $p_0$ and $n_0$ such that $X^*(P_0,n_0)>3.8$ and $p_0>10^4$. Then there exists an explicit constant $C=g(n_0,p_0)$ such that \[q_n \leq Cp^{\frac{1}{4}}(\log{p})^{\frac{n+1}{2}}\] for all $p \geq p_0$ and $n \leq n_0$.

\begin{proof}[Proof of~Corollary~\ref{C:1}]
This follows immediately from Theorem~\ref{T:1}, letting $u$ be the product of the first $n-1$ prime nonresidues. The fact that this holds for all $p \geq p_0$ and $n \leq n_0$ can be verified by showing that $g(n,p)$ is decreasing with $p$ and increasing with $n$ under the conditions given.
This is not hard to verify.  Indeed, calculus can be used to show that
for $X^*>3.8$, the expression\[1-\frac{\pi^2}{9}\frac{\log X^*+9}{3X^*}\] is increasing with $X^*$, and that $X^*$ increases with $p$ and decreases with $n$.  Similarly, the term\[\frac{\pi}{3}\sqrt{2e}\left(\frac{n}{n+1}\right)\sqrt{1+\frac{\sqrt{2}}{2B\log{p}-3}}\] is increasing with $n$ and decreasing with $p$.
\end{proof}

%\begin{comment}
\section*{Acknowledgements}
This research was completed as part of the Research Experience for Undergraduates and Teachers program at California State University, Chico funded
by the National Science Foundation (DMS-1559788).  We would also like to thank the anonymous referee for their helpful suggestions
which improved the quality of this paper.
%\end{comment}

\vspace{2ex}

{\footnotesize
\noindent
Shilin Ma\\
Carleton College\\
Carnegie Mellon University\\[0.5ex]

\noindent
Kevin J. McGown\\
California State University, Chico\\
University of New South Wales, Canberra\\[0.5ex]

\noindent
Devon Rhodes\\
California State University, Chico\\[0.5ex]

\noindent
Mathias Wanner\\
Villanova University\\
University of California, Santa Barbara\\[0.5ex]
}


\begin{thebibliography}{cc}
\bibitem{bib:Banks.Guo}
Banks, William D.; Guo, Victor Z.
\emph{Quadratic nonresidues below the Burgess bound.}
Int. J. Number Theory 13 (2017), no. 3, 751--759.
\bibitem{bib:Booker}
Booker, Andrew R.
\emph{On Mullin's second sequence of primes.}
Integers 12 (2012), no. 6, 1167--1177.
\bibitem{bib:Burgess1}
Burgess, D. A.
\emph{On character sums and primitive roots.}
Proc. London Math. Soc. (3) 12 (1962), 179--192.
\bibitem{bib:Burgess2}
Burgess, D. A.
\emph{A note on the distribution of residues and non-residues.}
J. London Math. Soc. 38 (1963), 253--256. 
\bibitem{bib:McGown.Lezowski}
Lezowski, Pierre; McGown, Kevin J.
\emph{The Euclidean algorithm in quintic and septic cyclic fields.}
Math. Comp. 86 (2017), no. 307, 2535--2549.
\bibitem{bib:McGown1}
McGown, Kevin J.
\emph{On the second smallest prime non-residue.}
J. Number Theory 133 (2013), no. 4, 1289--1299.
\bibitem{bib:McGown2}
McGown, Kevin J.
\emph{Norm-Euclidean cyclic fields of prime degree.}
Int. J. Number Theory 8 (2012), no. 1, 227--254.
\bibitem{bib:Norton}
Norton, Karl K.
\emph{Numbers with small prime factors, and the least kth power non-residue.}
Memoirs of the American Mathematical Society, No. 106 American Mathematical Society, Providence, R.I. 1971 ii+106 pp.
\bibitem{bib:Pollack}
Pollack, Paul.
\emph{Bounds for the first several prime character nonresidues.}
Proc. Amer. Math. Soc. 145 (2017), no. 7, 2815--2826.
\bibitem{bib:Pollack.Trevino}
Pollack, Paul; Trevi\~no, Enrique.
\emph{The primes that Euclid forgot.}
Amer. Math. Monthly 121 (2014), no. 5, 433--437. 
\bibitem{bib:Robbins}
Robbins, H.
\emph{A remark on Stirling's formula.}
Amer. Math. Monthly, 62 (1955), 26--29.
\bibitem{bib:Trevino2}
Trevi\~no, Enrique.
\emph{The Burgess inequality and the least kth power non-residue.}
Int. J. Number Theory 11 (2015), no. 5, 1653--1678.
\bibitem{bib:Trevino1}
Trevi\~no, Enrique.
\emph{The least k-th power non-residue.} 
J. Number Theory 149 (2015), 201--224.
\end{thebibliography}
\end{document}